\documentclass[12pt]{amsart}
\usepackage{amsmath,amssymb,enumerate,mathtools}
\newtheorem{theorem}{Theorem}[section]

\newtheorem{lemma}[theorem]{Lemma}

\newtheorem{corollary}[theorem]{Corollary}

\theoremstyle{definition}
\newtheorem{definition}[theorem]{Definition}

\newcommand{\bF}{\mathbb F}
\newcommand{\bR}{\mathbb R}

\newcommand{\bE}{\mathbb E}

\newcommand{\bZ}{\mathbb Z}
\newcommand{\bN}{\mathbb N}

\newcommand{\cB}{\mathcal{B}}

\newcommand{\cD}{\mathcal{D}}

\newcommand{\cL}{\mathcal{L}}

\newcommand{\cN}{\mathcal{N}}

\newcommand{\cV}{\mathcal{V}}

\newcommand{\nni}{\mathbb{Z}^+}
\newcommand{\eps}{\varepsilon}
\newcommand{\erdos}{Erd\H{o}s}

\DeclareMathOperator{\ex}{ex}
\DeclareMathOperator{\PG}{PG}

\DeclareMathOperator{\EX}{EX}

\DeclareMathOperator{\AG}{AG}
\DeclareMathOperator{\BB}{BB}

\newcommand{\del}{\!\setminus\!}

\begin{document}

\title{Stability and exact Tur\'an numbers for matroids}
\author[Liu]{Hong Liu}
\address{Mathematics Institute and DIMAP, University of Warwick, Coventry, CV4 7AL, UK.  Email address: {\tt h.liu.9@warwick.ac.uk}}
\author[Luo]{Sammy Luo}
\address{Department of Mathematics, Massachusetts Institute of Technology, Cambridge MA, USA. Email address: {\tt sammyluo@mit.edu}}
\author[Nelson]{Peter Nelson}
\address{Department of Combinatorics and Optimization, University of Waterloo, Waterloo, Canada. Email address: {\tt apnelson@uwaterloo.ca}}
\author[Nomoto]{Kazuhiro Nomoto}
\address{Department of Combinatorics and Optimization, University of Waterloo, Waterloo, Canada. Email address: {\tt 	knomoto@uwaterloo.ca}}
\thanks{The first author was supported by the Leverhulme Trust Early Career Fellowship~ECF-2016-523.}
\thanks{The third author was supported by a discovery grant from the Natural Sciences and Engineering Research Council of Canada}
\date{\today}
\begin{abstract}
	We consider the Tur\'an-type problem of bounding the size of a set $M \subseteq \bF_2^n$ that does not contain a linear copy of a given fixed set $N \subseteq \bF_2^k$, where $n$ is large compared to $k$. An \erdos-Stone type theorem [\ref{gn12}] in this setting gives a bound that is tight up to a $o(2^n)$ error term; our first main result gives a stability version of this theorem, showing that such an $M$ that is close in size to the upper bound in [\ref{gn12}] is close in edit distance to the obvious extremal example. Our second result shows that the error term in [\ref{gn12}] is exactly controlled by the solution to one of a class of `sparse' extremal problems, and in many cases eliminates the error term completely to give a sharp upper bound on $|M|$.
\end{abstract}

\maketitle

\section{Introduction}

This paper deals with Tur\'{a}n-type problems in the arithmetic setting of subsets of $\bF_2^n$, where we fix a set $N \subseteq \bF_2^k$ and consider, for $n$ much larger than $k$, the size of a set $M \subseteq \bF_2^n$ that does not contain any subset that is the image of $N$ under an injective linear map $\varphi \colon \bF_2^k \to \bF_2^n$. This is analogous to excluding a fixed subgraph $H$ from a graph $G$. Such problems have been considered both in the language of arithmetic combinatorics [\ref{bfhhl},\ref{gw10},\ref{hhl}] and equivalently matroid theory [\ref{gn14},\ref{gn16},\ref{tidor}]; here we will use the term `matroid' for brevity to describe the relevant notions of containment and isomorphism, as well as to highlight the strong analogies with graph theory, and to describe the `host' object $M$ and the `system of linear forms' being excluded in a unified way. While they are stated combinatorially, all of our results depend on a Fourier-analytic regularity lemma of Hatami et al [\ref{hhl}] and a new associated counting lemma due to the second author [\ref{luo}].

For an integer $n \ge 0$, we define a \emph{matroid} of \emph{dimension} $m$ to be a set $M \subset \bF_2^m\del\{0\}$. Write $\dim(M)$ for $m$. Given matroids $M \subset \bF_2^m$ and $N \subset \bF_2^n$,  A linear map $\varphi\colon \bF_2^n \to \bF_2^m$ is a \emph{homomorphism} from $N$ to $M$ if $\varphi(N) \subseteq M$. If $\varphi$ is bijective and $\varphi(N) = M$, then $\varphi$ is an \emph{isomorphism}. A subset of $M$ that is the image of $N$ under some injective linear map $\psi\colon \bF_2^n \to \bF_2^m$ is an \emph{$N$-restriction} of $M$. (Note that if $M$ is obtained by embedding a copy of $N$ in a higher dimensional space, then $|M| = |N|$ and $M$ contains $N$ as a restriction but $N$ does not contain $M$; this is analogous to adding isolated vertices to a graph). The \emph{rank} of a matroid $M \subset \bF_2^n$, written $r(M)$, is the smallest dimension of a subspace of $\bF_2^n$ containing $M$; clearly $r(M) \le \dim(M)$. 

Matroids in our sense essentially correspond to \emph{simple binary matroids} in the usual sense, except we distinguish between a matroid $M$ and a matroid $M^+$ obtained by embedding a copy of $M$ in a larger subspace. Our main results still hold if `matroid' is just read as `simple binary matroid'. In the other direction, stating that a matroid $M$ does not contain $N$ is equivalent to insisting that $M$ does not contain a nondegenerate copy of some `system of linear forms' in the language of [\ref{hhl}]. If $N = \{w_1,\dotsc,w_s\} \subseteq \bF_2^k$, then the corresponding system is $\cL_N = \{L_1, \dotsc, L_s\}$, where each linear form $L_i\colon (\bF_2^n)^k \to \bF_2^n$ is defined by $L_i(v_1, \dotsc,  v_k) = \sum_{j=1}^k w_{i,j}v_j$.

Write $\PG(t-1,2)$ for the rank-$t$ \emph{projective geometry}: this is the $t$-dimensional matroid $\bF_2^t \del \{0\}$ with $2^t-1$ elements. The \emph{critical number} of a matroid $M \subset \bF_2^n$, written $\chi(M)$, is the minimum $c$ for which $M$ has a homomorphism to $\PG(c-1,2)$.  Equivalently, $\chi(M)$ is the smallest codimension of a subspace of $\bF_2^n$ that is disjoint from $M$. This parameter, which is an analogue of chromatic number, is invariant under isomorphism and monotone with respect to containment. A matroid of critical number $1$ is \emph{affine}. The maximal rank-$t$ affine matroid, the complement of a hyperplane, is the rank-$t$ \emph{affine geometry}, denoted $\AG(t-1,2)$.

Given a set $\cN$ of matroids, we say $M$ is \emph{$\cN$-free} if $M$ has no restriction isomorphic to a matroid in $\cN$; write $\EX(\cN)$ for the class of $\cN$-free matroids. For each integer $n$, let $\ex(\cN,n)$ denote the `Tur\'an number' for $\cN$: the maximum size of an $n$-dimensional matroid in $\EX(\cN)$. An $n$-dimensional matroid in $\EX(\cN)$ of size $\ex(\cN,n)$ is \emph{extremal} in $\EX(\cN)$. A classical result of Bose and Burton gives this function exactly when $\cN$ consists of a single projective geometry. (Note that $\PG(t,2)$ has dimension $t+1$.)

\begin{theorem}\label{bbt}
	$\ex(\{\PG(t,2)\},n) = 2^n(1-2^{-t})$ for all $1 \le t < n$. 
\end{theorem}

Equality holds here for the \emph{Bose-Burton geometry} $\BB(n-1,2,t)$ of dimension $n$ and order $t$, defined as the complement of an $(n-t)$-dimensional subspace of $\bF_2^n$; this matroid has critical number $t$ and size $2^n-2^{n-t}$, and plays the role of a balanced complete multipartite graph. For each matroid $N$ of critical number $k$, the matroid $\BB(n-1,2,k-1)$ is $N$-free; the following analogue of the \erdos-Stone theorem shows that it is the largest $N$-free matroid up to an error term. 
\begin{theorem}[Matroidal Erd\H{o}s-Stone Theorem {[\ref{gn12}]}]\label{es}
	For every $N$,
	 \[\ex(\{N\},n) = 2^n(1-2^{1-\chi(N)}+o(1)).\]
\end{theorem}

 Analogously to excluding a bipartite subgraph, this statement is true but less descriptive when $N$ is affine, simply giving that $N$-free matroids are sparse. Such `sparse' extremal functions are known to be bounded above by $2^{\alpha n}$ for some $\alpha < 1$ (see [\ref{bq}, Lemma 21]), but, as in the case of graphs, are hard to asymptotically determine, even in the simplest nontrivial case where $N$ is a minimal linearly dependent set of four vectors in $\bF_2^3$. 
 
 When $\chi(N) > 1$, on the other hand, the above theorem asymptotically determines the function.
 Our first main result is a stability version of Theorem~\ref{es}, showing that any $N$-free matroid that is close in size to $\ex(\{N\},n)$ is close in edit distance to a Bose-Burton geometry.  

\begin{theorem}\label{esstability}
	Let $N$ be a matroid. For all $\delta > 0$ there exists $\eps > 0$ such that, if $M$ is an $n$-dimensional, $N$-free matroid with \[||M| - (1-2^{1-\chi(N)})2^n| \le \eps 2^n,\] then $|M \Delta B| < \delta 2^n$ for some matroid $B \cong BB(n-1,2,\chi(N)-1)$.
\end{theorem}

Our other main result is Theorem~\ref{maintech}, which is somewhat technical to state here, but in the spirit of [\ref{s97}] reduces the problem of finding $\ex(\cN,n)$ for arbitary $\cN$ to the sparse case where $\cN$ contains an affine matroid. Strikingly, the theorem also gives a way to determine Tur\'an numbers exactly in a much wider range of cases than is currently possible in the graphical setting. We defer the precise statement of Theorem~\ref{maintech} to the final section, instead giving here some nice corollaries.  The first, analogous to the celebrated `critical edge theorem' of Simonovits [\ref{s74}], gives a range of cases where the bound given by Bose-Burton geometries is eventually exact. 

\begin{theorem}\label{criticalcritical}
	Let $N$ be a matroid such that $\chi(N \del e) < \chi(N)$ for some $e \in N$. Then $\ex(\{N\},n) = 2^n(1-2^{1-\chi(N)})$ for all sufficiently large $n$, and the function is attained only by Bose-Burton geometries.
\end{theorem}

We also obtain  exact results for many `graphic' matroids; for a simple graph $G = (V,E)$, write $M(G)$ for the matroid comprising the set of columns of the $V \times E$ incidence matrix of $G$. These objects are widely studied in the combinatorial theory of matroids; in fact, the critical number of $M(G)$ satisfies $\chi(M(G)) = \lceil \log_2 \chi(G) \rceil$, where $\chi(G)$ is the chromatic number of $G$. The theorem below summarises for which graphs our machinery gives an exact result.

\begin{theorem}\label{graphs}
	Let $G$ be a graph with $\chi(G)\geq 2$ and let $t$ be the unique integer for which $2^t < \chi(G) \le 2^{t+1}$. If there is a forest $F$ of $G$ with $\chi(G -F) \le 2^t$, then $\ex(\{M(G)\},n) = 2^n(1-2^{-t}) + \ell$ for all sufficiently large $n$, where $\ell < 2^{|F|-1}$ is an integer depending only on $G$.
\end{theorem}
 
 We remark that if $G$ is a nonbipartite cubic graph, then $t = 1$ and, since $E(G)$ is the union of a maximum cut and a matching, the required $F$ exists, and we get $\ex(\{M(G)\},n) = 2^{n-1} + \ell$ for large $n$, where $\ell$ is constant. In fact, in this case, $\ell = 2^{\nu-1} - 1$, where $\nu$ is the size of a matching whose complement is a maximum cut.
 
 More generally, the dependence of $\ell$ on $G$ is implicit in Corollary~\ref{finiteproblem}, and while technical, we can compute $\ell$ exactly in many natural cases. For $t \ge 2$, let $K_t$ denote the complete graph on $t$ vertices. 

\begin{theorem}\label{cliques}
	For $t\geq 2$, $\ex(\{M(K_t)\},n) = 2^n(1 - 2^{-t_0}) + 2^{t-2^{t_0}-1}-1$ for all sufficiently large $n$, where $t_0$ is the largest integer with $2^{t_0} < t$.
\end{theorem}

Let $O_6$ denote the octahedral graph. The maximum number of edges in a large $O_6$-free graph depends on the extremal function for $C_4$-free graphs (see [\ref{s97}]) and is not known precisely, but interestingly, we obtain a precise result in our setting. 

\begin{theorem}\label{oct}
	$\ex(\{M(O_6)\},n) = 2^{n-1} + 4$ for all sufficiently large $n$. 
\end{theorem}

We remark that Corollary~\ref{finiteproblem} also gives (in principle) a similar exact result for the icosahedral graph. The tetrahedron and dodecahedron are covered by the previous two theorems, while the cube has affine cycle matroid and so falls into the difficult sparse regime.

\section{Regularity and Counting}
Our arguments use the following lemma, a direct consequence of recently-developed regularity results for binary matroids. It is essentially proven in the course of the proof of~[\ref{luo}, Theorem~4.1]. We give the proof explicitly here for the sake of completeness.

\begin{lemma}[Weak counting lemma]\label{blowup}
	For all $k \in \nni$ and $\delta > 0$, there is some $n_0 \in \nni$ and some $\alpha > 0$ such that, for every $n \ge n_0$ and matroid $M \subseteq \bF_2^n$, there is a matroid $M_0 \subseteq \bF_2^n$ such that 
	\begin{itemize}
		\item $|M-M_0| \le \delta 2^{r(M)}$, and
		\item if $N$ is a matroid with dimension at most $k$ that is homomorphic to $M_0$, then $M$ has at least $\alpha (2^n)^{r(N)}$ distinct $N$-restrictions. 
	\end{itemize}
\end{lemma}
Note that the first property implies in particular that $|M_0| \ge |M|-\delta 2^{r(M)}$. We note that a counting version of the Matroidal Erd\H{o}s-Stone theorem follows immediately from this lemma from simply applying Theorem~\ref{bbt} to $M_0$.
\begin{corollary}[Counting Matroidal Erd\H{o}s-Stone Theorem]
For any matroid $N$ and $\delta>0$, there is some $n_0 \in \nni$ and some $\alpha > 0$ such that if a matroid $M$ with $r(M)\ge n_0$ satisfies
\[
|M|\geq (1-2^{1-\chi(N)}+\delta)2^{r(M)},
\]
then $M$ has at least $\alpha 2^{r(N)}$ distinct $N$-restrictions.
\end{corollary}

The rest of this section gives a brief overview of the technical machinery required to prove Lemma~\ref{blowup}, which is expounded upon in more detail in [\ref{luo}]. The reader may skip this section without missing any information necessary for understanding the remainder of the paper.

The first result we state is the Strong Decomposition Theorem of~[\ref{bfhhl}], a regularity result on binary matroids analogous to the Szemer\'{e}di regularity lemma for graphs. Stating the result precisely requires the introduction of several technical definitions not useful in the rest of the paper. We omit several definitions not relevant in the proofs given below. For example, in lieu of a precise definition, it suffices to think of non-classical polynomials as a particular class of functions from $\bF_2^n$ to $\bR/\bZ$. Likewise, we do not need the definition of the Gowers norms $\|f\|_{U^d}$, nor what it means for a polynomial factor to be $r$-regular. We refer the reader to~[\ref{bfhhl}] for more details.

\begin{definition}
A \emph{(non-classical) polynomial factor} $\mathcal{B}$ of $\bF_2^n$ of \emph{complexity} $C$ is a partition of $\bF_2^n$ into finitely many pieces, called \emph{atoms}, such that for some (non-classical) polynomials $P_1,\dots,P_C$, each atom is the solution set $\{x\colon (P_1(x),P_2(x), \dotsc, P_C(x)) = (b_1, \dotsc, b_C)\}$ for some $b=(b_1,\dots,b_C)\in (\bR/\bZ)^C$. The \emph{degree} of $\mathcal{B}$ is the highest degree among $P_1, \dotsc, P_C$.
\end{definition}

The next theorem states that any function $\bF_2^n \to \{0,1\}$ (which we think of as the indicator function for a given matroid $M$), can be decomposed as the sum of three functions; the function $f_1$ just encodes the density of $f$, while $f_2$ is `highly structured' in the sense of having a small Gowers Uniformity norm, and $f_3$ is a small $L^2$ error term.  

\begin{theorem}[Strong Decomposition Theorem, {[\ref{bfhhl}, Theorem~5.1]}]\label{decomp}
Suppose $\delta > 0$ and $ d \geq 1$ is an integer. Let $\eta\colon \bN
\to \bR^+$ be an arbitrary non-increasing function and $r\colon \bN \to \bN$ be an arbitrary
non-decreasing function. Then there exist $n_1 =
n_1(\delta, \eta, r, d)$ and $C =
C(\delta,\eta,r,d)$ such that the following holds.

For all $f: \bF_2^n \to \{0,1\}$ with $n > n_1$, there
exist three functions $f_1, f_2, f_3: \bF_2^n \to
\bR$ and a polynomial factor  $\cB$ of
degree at most $d$ and complexity at most $C$ such that the following conditions hold:
\begin{itemize}
\item[(i)]
$f=f_1+f_2+f_3$.
\item[(ii)]
$f_1 = \bE([f|\cB])$, the expected value of $f$ on an atom of $\cB$.
\item[(iii)]
$\|f_2\|_{U^{d+1}} \leq \eta(|\cB|)$.
\item[(iv)]
$\|f_3\|_2 \leq \delta$.
\item[(v)]
$f_1$ and $f_1 + f_3$ have range $[0,1]$; $f_2$ and $f_3$ have range $[-1,1]$.
\item[(vi)]
$\cB$ is $r$-regular.
\end{itemize}
\end{theorem}

In analogy with the terminology for the Szemer\'{e}di regularity lemma, we will call a decomposition of $f=1_M$ with the properties given by Theorem~\ref{decomp} for parameters $\delta, \eta, r, d$ a \emph{$(\delta,\eta,r,d)$-regular partition of $f$ (or of $M$)}, and we say that $\cB$ is its \emph{corresponding factor}. Similarly, an \emph{$(\eta,r,d)$-regular partition of $f$ (or of $M$)} is the same thing with an unspecified value for $\delta$.

A corresponding counting lemma, analogous to the counting lemma for graphs, is proven in [\ref{luo}]. It is best stated in terms of a structure called a reduced matroid.

\begin{definition}[Reduced Matroid]
Given a matroid $M\subseteq \bF_2^n\setminus \{0\}$ and an $(\eta,r,d)$-regular partition $f_1+f_2+f_3$ of $M$ with corresponding factor $\cB$, for any $\eps,\zeta>0$ define the $(\eps,\zeta)$-\emph{reduced matroid} $R=R_{\eps, \zeta}$ to be the subset of $\bF_2^n$ whose indicator function is constant on each atom $b$ of $\cB$ and equals $1$ on $b$ if and only if
\begin{enumerate}
    \item $\bE[|f_3(x)|^2\mid x\in b]\leq \eps^2$, and
    \item $\bE[f(x)\mid x\in b]\geq \zeta$.
\end{enumerate}
\end{definition}

So, $R$ is the union of the atoms of the decomposition in which $M$ has high density and the $L^2$ error term is small. The next lemma shows that any small matroid having a homomorphism to the reduced matroid $R_{\eps,\zeta}$ is found many times as a restriction in $M$ itself. 

\begin{lemma}[Counting Lemma, {[\ref{luo}, Theorem~3.13]}]\label{counting}
For every matroid $N$, positive real number $\zeta$, and integer $d\geq |N|-2$, there exist positive real numbers $\beta$ and $\eps_0$, a positive nonincreasing function $\eta\colon \bZ^+\to \bR^+$, and positive nondecreasing functions $r,\nu\colon\bZ^+\to \bZ^+$ such that for all $\eps \leq \eps_0$, if $f_1+f_2+f_3 = 1_M$ is an $(\eta,r,d)$-regular partition of an $n$-dimensional matroid $M$ with corresponding factor $\cB$ satisfying $\nu(|\cB|) \le n$, and $N$ is homomorphic to the reduced matroid $R_{\eps,\zeta}$, then $M$ has at least $\beta \frac{(2^n)^{r(N)}}{\|\cB\|^{|N|}}$ distinct $N$-restrictions.

\end{lemma}

From this, Lemma~\ref{blowup} follows readily.

\begin{proof}[Proof of Lemma~\ref{blowup}]
Given $k$ and $\delta$, let $d=2^k$, so $d >  |N|$ for every matroid $N$ of dimension at most $k$. Let $\delta'=\frac{\delta}{2}$.
By Lemma~\ref{counting}, there exist $\beta, \eta, \eps_0, r, \nu$ such that if, for some $\eps\leq \eps_0$, $N$ is homomorphic to the reduced matroid $R=R_{\eps,\delta'}$ given by an $(\eta,r,d)$-regular partition of an $n$-dimensional matroid $M$ with corresponding factor $\cB$ satisfying $r(M)\geq \nu(|\cB|)$, then $M$ has at least $\beta \frac{(2^n)^{r(N)}}{\|\cB\|^{|N|}}$ distinct $N$-restrictions. Fix such a choice of $\beta,\eta,\varepsilon,r,\nu$.

By Theorem~\ref{decomp}, for large enough $n$
we have a $(\varepsilon \delta'^{1/2},\eta,r,d)$-regular partition $1_M = f_1+f_2+f_3$ of $M$, whose corresponding factor $\cB$ has complexity $|\cB|\leq C$, where $C$ depends on only $\delta',\varepsilon, \eta,r,d$. Since there are finitely many non-isomorphic matroids $N$ of dimension at most $k$, we can then take $C$ to depend only on $\delta$ and $k$.

Let $n_0=\max(n_1,\nu(C))$, and suppose $r(M)=n\geq n_0$. Let $M_0=R_{\eps,\delta'}$, with the partition given as above. By the argument above, $M_0$ then satisfies the second claim of Lemma~\ref{blowup}, since $\|\cB\|^{|N|}$ is bounded for fixed $\delta$ and $k$.
It suffices to show that $M_0$ also satisfies the first claim of Lemma~\ref{blowup}. Observe that the only elements in $M - M_0$ are either
\begin{itemize}
    \item[(i)] In an atom $b$ of $\cB$ such that $E[|f_3(x)|^2\mid x\in b]>\varepsilon^2$, or
    \item[(ii)] In an atom $b$ of $\cB$ such that $E[f(x)\mid x\in b]<\delta'$.
\end{itemize}
Let $S$ be the union of the atoms $b$ of $\cB$ such that $\bE[|f_3(x)|^2\mid x\in b]>\varepsilon^2$. Then by condition (iv) of Theorem~\ref{decomp},
\[
\varepsilon^2 \delta' \geq \|f_3\|_2^2=\bE_x[|f_3(x)|^2] \geq \frac{|S|}{2^n} \varepsilon^2,
\]
so $|S|\leq \delta' 2^n$. Likewise, let $T$ be the subset of $\bF_2^n$ contained in atoms $b$ of $\cB$ such that $\bE[f(x)\mid x\in b]<\delta'$. Then $|T\cap M|< \delta' |T|\leq \delta' 2^n$.
So, $|M-M_0|\leq \delta 2^n$, as required.
\end{proof}

\section{Stability}

The following result is an analogue of the Andrasfai-Erd\H{o}s-S\'os theorem, proved by Govaerts and Storme [\ref{gs}] (also, see [\ref{g14}]).  

\begin{theorem}\label{aes}
	Let $t \ge 2$ be an integer. If $r \ge t+2$ and $M$ is a rank-$r$, $\PG(t-1,2)$-free matroid with $|M| > 2^{r} (1-2^{1-t}-3 \cdot 2^{-2-t})$, then $\chi(M) \le t-1$. 
\end{theorem}

Combined with the material in the last section, this is enough to prove Theorem~\ref{esstability}, which we restate here.

\begin{theorem}\label{stability}
	Let $N$ be a matroid. For all $\eps > 0$ there exists $\delta > 0$ and $r_0 \in \nni$ such that, if $r \ge r_0$ and $M$ is a rank-$r$, $N$-free matroid with $|M| \ge (1-2^{1-\chi(N)} - \delta)2^{r}$, then $|M \Delta B| < \eps 2^{r}$ for some matroid $B \cong BB(r-1,2,\chi(N)-1)$.
\end{theorem}
\begin{proof}
	Let $N$ be a matroid with $\chi(N) = t$. Let $\eps > 0$ and let $\delta = \min(\tfrac{1}{7}\eps, 2^{-2-t})$. Let $n_0$ be given by Lemma~\ref{blowup}, invoked with $\delta$ and with $k = r(N)$, and let $r_0 = \max(t+2,n_0)$.  Let $r \ge r_0$ and $M$ be a rank-$r$ $N$-free matroid with $|M| \ge (1-2^{1-t}-\delta)2^r$ . Let $M_0$ by given by Lemma~\ref{blowup}. So $|M_0| \ge (1-2^{1-t}-2\delta)2^r > (1-2^{1-t}-3\cdot 2^{-2-t})2^r$. If $M_0$ has a $\PG(t-1,2)$-restriction, then since $N$ is homomorphic to $\PG(t-1,2)$, the matroid $M$ has an $N$-restriction, a contradiction. Thus $M_0$ has no $\PG(t-1,2)$-restriction, so by Theorem~\ref{aes}, $M_0$ is a restriction of some $B \cong \BB(r-1,2,t-1)$, giving $|B-M_0| \le 2\delta 2^r$.  But $|M-M_0| \le \delta 2^{r}$; it follows that $|M - B| \le 3\delta 2^r$. Now
	\[(1-2^{1-t}-\delta)2^r < |M| = |M \cap B| + |M-B| \le |M \cap B| + 3\delta 2^r,\] so $|M \cap B| \ge (1-2^{1-t}-4\delta)2^r = |B|-(4\delta)2^r$, so $|B-M| \le (4\delta)2^r$. Therefore $|M \Delta B| \le (7\delta)2^r < \eps 2^r$, as required.
\end{proof}

\section{Decomposition Families}

	If $N$ is a matroid with $\dim(N) \le n-t$, then let $N^{n,t} \subseteq \bF_2^n$ denote the unique (up to isomorphism) rank-$n$ matroid for which there is a $t$-codimensional subspace $F$ such that $N^{n,t}\cap F$ is an $N$-restriction of $M$, and $\bF_2^n-F \subseteq N^{n,t}$. Thus $N^{n,t}$, which is essentially obtained by placing a copy of $N$ inside the empty $t$-codimensional flat of $\BB(n-1,2,t)$, contains $\BB(n-1,2,t)$ and is contained in $\PG(n-1,2)$. 

	\begin{definition}Let $\cN$ be a set of matroids and $k = \min_{N \in \cN}\chi(N)-1$. Note that $\dim(N) > k$ for all $N \in \cN$. The \emph{decomposition family} $\cD(\cN)$ is the collection of (isomorphism classes of) restriction-minimal matroids $D$ for which some $N \in \cN$ is contained in $D^{\dim(N),k}$.
	\end{definition}	
	 Equivalently, a matroid $D$ is in $\cD$ if $D$ is restriction-minimal subject to being a restriction of the intersection of some $N \in \cN$ with a $k$-codimensional subspace. A third definition is that $\cD$ is the set of restriction-minimal matroids $D$ for which removing some $D$-restriction from some $N \in \cN$ drops the critical number of $N$ to $k$. 
	
	 Note that if $\cN$ is finite, then so is $\cD$. Note also that, if $N \in \cN$ has $\chi(N) = k+1$, then $\AG(\dim(N)-k-1,2)^{\dim(N),k} \cong \BB(\dim(N)-1,2,k+1)$, which has an $N$-restriction; it follows that $\cD$ must contain an affine matroid. Finally, observe that by minimality, every $D \in \cD$ satisfies $r(D) = \dim(D)$.
	
	\begin{theorem}\label{maintech}
		Let $\cN$ be a finite set of matroids and define  $k = \min_{N \in \cN}\chi(N)-1$. Then, for all sufficiently large $n$, \[\ex(\cN,n) = 2^n(1-2^{-k})+\ex(\cD(\cN),n-k)\] and the extremal $n$-dimensional matroids in $\EX(\cN)$ are those of the form $M \cong (M_0)^{n,k}$ for some $(n-k)$-dimensional matroid $M_0$ that is extremal in $\EX(\cD(\cN))$.
	\end{theorem}
	\begin{proof}
		If $k = 0$ then $\cD(\cN) = \cN$ and the result is trivial; assume that $k \ge 1$. We first show that the formula is a lower bound. Indeed, let $M_0 \in \EX(\cD(\cN))$ be extremal with $\dim(M_0) = n-k$, and let $M \cong (M_0)^{n,k}$, where $F \subseteq \bF_2^n$ is the $(n-k)$-dimensional subspace for which $M \cap F$ is an $M_0$-restriction of $M$.  For any restriction $N'$ of $M$ with $\chi(N') \ge k+1$, the intersection $N' \cap F$ contains a $D'$-restriction of $N'$ for some $D' \in \cD(\{N'\})$, but $M_0$ has an $N' \cap F$-restriction, so $N' \cap F$ contains no $D$-restriction for any $D \in \cD(\cN)$. It follows that there is a matroid in $\cD(\{N'\})$ containing no restriction in $\cD(\cN)$. If $N' \in \cN$, this is a contradiction; thus, $M \in \EX(\cN)$ and so $\ex(\cN,n) \ge |M| = 2^n - 2^{n-k} + |M_0| = 2^n(1-2^{-k}) + \ex(\cD(\cN),n-k)$.

	 	Let $N_0 \in \cN$ satisfy $\chi(N_0) = k+1$ and let $d_0 = \max_{D \in \cD(\cN)}\dim(D)$ and $\eps = \tfrac{1}{2}(2^{d_0+k}-2^{d_0})^{-1}$, and let $r_0$ and $\delta$ be given by Theorem~\ref{stability} for $N_0$ and $\eps$. Suppose now that $M \in \EX(\cN)$ satisfies $\dim(M) = n \ge r_0$ and $|M| \ge 2^n(1-2^{-k}) + \ex(\cD(\cN),n-k)$. Since $M$ is $N_0$-free and satisfies $|M| \ge (1-2^{1-\chi(N_0)})2^n$, Theorem~\ref{stability} gives that $|M \Delta B| < \eps 2^n$ for some matroid $B \cong BB(n-1,2,k)$. So $\bF_2^n$ has an $(n-k)$-dimensional subspace $W$ with $|M \del (\bF_2^n-W)| < \eps 2^n \le 2\eps(2^n-2^{n-k})$. 
		 
		Suppose that $M \cap W$ has a $D$-restriction for some $D \in \cD(\cN)$, and let $N \in \cN$ be such that $D^{\dim(N),k}$ has an $N$-restriction, with $d = \dim(N)$. Let $V_0 \subseteq W$ be a $d$-dimensional subspace containing a $D$-restriction of $M$, and $\cV$ be the collection of $(d+k)$-dimensional subspaces $V$ of $\bF_2^n$ for which $V \cap W = V_0$. Note that each $V \in \cV$ contains precisely $2^{d+k}-2^d$ elements of $\bF_2^n-W$, and that each $x \in \bF_2^n-W$ is in the same number of subspaces $V \in \cV$. Since $|M \del (\bF_2^n-W)| < 2\eps (2^n-2^{n-k})$, an averaging argument gives that some $V_1 \in \cV$ satisfies $|(V_1-V_0) \cap M| > (2^{d+k}-2^d)(1-2\eps)$; using $d \le d_0$ and the choice of $\eps$, this implies that $|(V_1-V_0) \cap M| > 2^{d+k}-2^d-1$ and so $V_1-V_0 \subseteq M$. It follows that $M \cap V_1$ has a $D^{d+k,k}$-restriction and so has an $N$-restriction, a contradiction. Therefore $M \cap W \in \EX(\cD(\cN))$.
		 
		Let $M_0$ be the $(n-k)$-dimensional matroid $\varphi(M \cap W)$, where $\varphi \colon W \to \bF_2^{n-k}$ is an isomorphism. Note that $M_0 \in \EX(\cD(\cN))$ and that $M$ is a restriction of $(M_0)^{n,k}$. Thus $|M| \le 2^n - 2^{n-k} + |M_0| \le 2^n(1-2^{-k}) + \ex(\cD(\cN),n-k)$, with equality precisely when $M = (M_0)^{n,k}$ and $M_0$ is extremal in $\EX(\cD(\cN),n-k)$. This gives the result. 	 
	\end{proof}

	In general, since $\cD(\cN)$ always contains some affine matroid, the function $\ex(\cD(\cN),n)$ is difficult to compute. The analogous problem in graph theory also has the same flavour, reducing arbitrary Tur\'an problems to the difficult `sparse' ones where bipartite graphs are excluded. However, in many natural cases in the matroidal setting, $\ex(\cD(\cN),n)$ can easily be determined exactly, or reduced to a finite computation. 

	An $n$-dimensional matroid is \emph{free} if it is a basis of $\bF_2^n$; all such $n$-dimensional matroids are isomorphic. Note that any matroid of rank at least $r$ contains an $r$-element free restriction. Thus, $\cD(\cN)$ contains at most one free matroid $I$. If $\cD$ consists of just a single free matroid $I$, we trivially have $\ex(\cD,n) = 2^{|I|-1}-1$ for all $n \ge |I|-1$, with equality only for projective geometries. The family $\cD(\cN)$ has this form precisely when some $N \in \cN$ can have its critical number dropped to $k$ by removing a linearly independent set $I$, while removing any linearly dependent set of at most $|I|$ from any $N \in \cN$ cannot achieve this. 
	
	\begin{corollary}
		Let $\cN$ be a finite collection of matroids and let $k = \min_{N \in \cN}(\chi(N))-1$. If there is some $t > 0$ such that 
		\begin{enumerate}
			\item\label{inddrop} there is a $t$-element linearly independent subset $I$ of some $N \in \cN$ for which $\chi(N \del I) = k$, and
			\item\label{justind} $\chi(N' \del X) \ge k+1$ for every $N' \in \cN$ and each linearly dependent set $X \subseteq N'$ with $|X| \le t$, 
		\end{enumerate}
	then $\ex(\cN,n) = 2^n(1-2^{-k}) + 2^{t-1}-1$ for all sufficiently large $n$, with equality precisely for matroids of the form $\PG(t-2,2)^{n,k}$.
	\end{corollary}

	Note that Theorem~\ref{criticalcritical} follows immediately from the above result. For a graph $G$, it is well-known (see [\ref{oxley} p.589]) that $\chi(M(G)) = \lceil \log_2 \chi(G) \rceil$, where $\chi(G)$ is the chromatic number, and moreover that a subset $X$ of $M(G)$ is linearly independent if and only if the corresponding set of edges in $G$ is acyclic. If $2^{k} < s \le 2^{k+1}$, we thus have $\chi(M(K_s)) = k+1$. It is routine to check that removing a matching of size $s-2^k$ from $K_s$ drops the chromatic number to $2^k$ (and hence the critical number to $k$) and that removing any other set of $(s-2^k)$ edges does not drop the chromatic number as far. These facts, combined with the above corollary, imply Theorem~\ref{cliques}. 
	
	If (\ref{inddrop}) above holds for some $t$ but (\ref{justind}) does not, then things are still fairly nice. Condition (\ref{inddrop}) guarantees that $\cD(\cN)$ contains a free matroid of size at most $t$, and thus that $\EX(\cD(\cN))$ contains nothing of rank at least $t$, so $\ex(\cD(N),n)$ is constant for all $n \ge t$, and can be computed by just considering the finitely many nonisomorphic matroids of rank at most $t-1$.
	
	\begin{corollary}\label{finiteproblem}
		Let $\cN$ be a finite set of matroids such that some $N_0 \in \cN$ has a linearly independent set $I$ for which $\chi(N_0 \del I) \le k$, where  $k = \displaystyle \min_{N \in \cN}(\chi(N))-1$; let $t$ be the size of a smallest such $I$. Then 
		\[\ex(\{N\},n) = 2^n(1-2^{-k}) + \ex(\cD(\cN),t-1)\]
		 for all sufficiently large $n$.
	\end{corollary}

	 Using $\chi(M(G)) = \lceil \log_2 \chi(G) \rceil$, the above directly implies Theorem~\ref{graphs}. In specific cases, we can compute $\ex(\cD(\cN),t-1)$ explicitly; for example, the matroid $M(O_6)$ of the octahedron has critical number $2$, and satisfies $\cD(\{M(O_6)\}) = \{I_4,C_4\}$, where $I_4$ is a basis for $\bF_2^4$, and $C_4$ is a minimal linearly dependent set of four vectors in $\bF_2^3$. If $M \in \EX(\{I_4,C_4\})$ then $r(M) \le 3$ and it is easy to check that $|M| \le 4$. Thus $\ex(\{M(O_6)\},n) = 2^n-2^{n-1}+ 4 = 2^{n-1}+4$ for all large $n$, giving Theorem~\ref{oct}.

\section*{References}
\newcounter{refs}
\begin{list}{[\arabic{refs}]}
{\usecounter{refs}\setlength{\leftmargin}{10mm}\setlength{\itemsep}{0mm}}

\item \label{bfhhl}
A. Bhattacharyya, E. Fischer, H. Hatami, P. Hatami, S. Lovett,
Every locally characterized affine-invariant property is
testable,
Proceedings of the forty-fifth annual ACM symposium on Theory
of computing (2013), ACM, 429--436.

\item \label{bq}
J.E. Bonin, H. Qin,
Size functions of subgeometry-closed classes of
representable combinatorial geometries,
Discrete Math. 224, (2000) 37--60.

\item\label{bb}
R. C. Bose, R. C. Burton, 
A characterization of flat spaces in a finite geometry and the uniqueness of the Hamming and the MacDonald codes, 
J. Combin. Theory 1 (1966), 96--104. 

\item\label{g14}
J. Geelen,
A geometric version of the Andrasfai-\erdos-S\'os theorem,
arXiv:1401.5769

\item\label{gn12}
J. Geelen, P. Nelson, 
An analogue of the Erd\H os-Stone theorem for finite geometries, 
Combinatorica 35 (2015), 209--214. 

\item\label{gn14}
J. Geelen, P. Nelson, 
Odd circuits in dense binary matroids, 
Combinatorica 37 (2017), 41--47. 

\item\label{gn16}
J. Geelen, P. Nelson, 
The critical number of dense triangle-free binary matroids,
J. Combin. Theory Ser. B 116(2016), 238--249.

\item\label{gs}
P. Govaerts, L. Storme,
The classification of the smallest nontrivial blocking sets in $\PG(n,2)$, 
J. Combin. Theory Ser. A 113 (2006), 1543--1548.

\item\label{gw10}
W. T. Gowers, J. Wolf, 
The true complexity of a system of linear equations, 
Proc. London Math. Soc. 100 (2010), 155--176. 

\item\label{hhl}
H. Hatami, P. Hatami, S. Lovett,
General systems of linear forms: Equidistribution and true complexity,
Adv. Math. 292 (2016), 446--477. 

\item\label{luo}
S. Luo, 
A counting lemma for binary matroids and applications to extremal problems,
arXiv:1610.09587 [math.CO]

\item \label{oxley}
J. G. Oxley, 
Matroid Theory,
Oxford University Press, New York (2011).

\item\label{s74}
M. Simonovits, 
Extremal graph problems with symmetrical extremal graphs, additional chromatic conditions, 
Discrete Math. 7 (1974), 349--376.

\item\label{s97}
M. Simonovits, 
How to solve a Tur\'an type extremal graph problem? (linear decomposition), Contemporary trends in discrete mathematics (Stirin Castle, 1997), pp. 283–305, DIMACS Ser. Discrete Math. Theoret. Comput. Sci., 49, Amer. Math. Soc., Providence, RI, 1999.

\item\label{tidor}
J. Tidor, 
Dense binary $\PG(t-1,2)$-free matroids have critical number $t-1$ or $t$, 
arXiv:1508.07278 
\end{list}

\end{document}